\documentclass[a4paper,11pt,onecolumn]{amsart}
\usepackage{amsfonts}
\usepackage{pb-diagram}
\title[Actions on nilpotent covers]{The action of mapping classes on nilpotent covers of surfaces}

\newtheorem{thm}{Theorem}
\newtheorem{lemma}[thm]{Lemma}

\newtheorem{conje}[thm]{Conjecture}

\newtheorem{cor}[thm]{Corollary}
\newtheorem{prop}[thm]{Proposition}

\numberwithin{thm}{section}

\newcommand{\al}{\alpha}

\newcommand{\gam}{\Gamma}

\newcommand{\bZ}{\mathbb{Z}}
\newcommand{\bR}{\mathbb{R}}
\newcommand{\bC}{\mathbb{C}}

\DeclareMathOperator{\Aut}{Aut}
\DeclareMathOperator{\Out}{Out}

\DeclareMathOperator{\Mod}{Mod}

\DeclareMathOperator{\rk}{rk}

\newcommand{\bQ}{\mathbb{Q}}

\newcommand{\bH}{\mathbb{H}}
\newcommand{\whG}{\widehat{G}}

\newcommand{\mL}{\mathcal{L}}

\author[T. Koberda]{Thomas Koberda}
\address{Department of Mathematics\\ Harvard University\\ 1 Oxford St.\\ Cambridge, MA 02138 }
\email{ koberda@math.harvard.edu}
\begin{document}
\begin{abstract}
Let $\Sigma$ be a surface whose interior admits a hyperbolic structure of finite volume.  In this paper, we show that any infinite order mapping class acts with infinite order on the homology of some universal $k$--step nilpotent cover of $\Sigma$.  We show that a Torelli mapping class either acts with infinite order on the homology of a finite abelian cover, or the suspension of the mapping class is a $3$--manifold whose fundamental group has positive homology gradient.  In the latter case, it follows that the suspended $3$--manifold has a large fundamental group.  It follows that every element of the Magnus kernel suspends to give a $3$--manifold with a large fundamental group.
\end{abstract}
\maketitle
\begin{center}
\today
\end{center}
\tableofcontents
\section{Introduction}
Let $\Sigma$ be a surface whose interior admits a hyperbolic structure of finite volume and let $p\in\Sigma$ be a marked point in the interior of $\Sigma$.  The object of study in this paper is the group $\Mod(\Sigma,p)$, the mapping class group of $\Sigma$ preserving the marked point.  Recall that this is the group of marked point preserving diffeomorphisms of $\Sigma$ up to isotopies which fix the marked point.  Thinking of $p$ as a basepoint for the fundamental group $\pi_1(\Sigma)$, we may identify $\Mod(\Sigma,p)$ with a subgroup of $\Aut(\pi_1(\Sigma))$.  In this way, we can consider the action of $\Mod(\Sigma,p)$ on characteristic subgroups, quotients, and quotients of subgroups of $\pi_1(\Sigma)$.

In \cite{K1}, the author showed that if $\psi\in\Mod(\Sigma,p)$ is any nontrivial mapping class, then $\psi$ acts nontrivially on the homology of some finite cover of $\Sigma$.  If $\{K_i\}$ is any exhausting sequence of finite index characteristic subgroups of $\pi_1(\Sigma)$, one can show that $\psi$ acts nontrivially on $K_i^{ab}$ for some $i$.  From the data of the actions of $\psi$ on the homology of finite covers of $\Sigma$, it is also possible to recover the Nielsen--Thurston classification of $\psi$.

In \cite{K1}, the action of $\psi$ on the homology of each $K_i$ might be finite order.  It seems highly unlikely that an infinite order mapping class should act with finite order on the homology of every finite cover.  From \cite{K1}, we have the following:
\begin{prop}
Let $T$ be a Dehn twist about an essential simple closed curve on $\Sigma$, or more generally a multitwist about an essential multicurve.  Then there exists a finite abelian cover $\Sigma'\to\Sigma$ such that each lift of $T$ acts with infinite order on $H_1(\Sigma',\bZ)$.
\end{prop}
Pseudo-Anosov homeomorphisms are much more difficult to describe, and the analogue of the previous proposition is unknown for those mapping classes:
\begin{conje}\label{c:rho}
Let $\psi$ be a mapping class which is pseudo-Anosov when restricted to some connected subsurface of $\Sigma$.  Then there exists a finite cover $\Sigma'\to\Sigma$ such that each lift of $\psi$ acts with spectral radius strictly greater than $1$ (and in particular with infinite order) on the real homology of $\Sigma'$.
\end{conje}
In this paper, we shall build on the ideas of \cite{K1} and prove a non--exclusive dichotomy about mapping classes which provides evidence supporting Conjecture \ref{c:rho}.  Recall that the {\bf suspension} of a mapping class is the mapping torus of some representative of the mapping class.  It is easy to see that the homotopy type of the suspension of a homeomorphism depends only on the mapping class of the homeomorphism.  Recall that a group $G$ is called {\bf large} if there is a finite index subgroup $G'<G$ which surjects to a nonabelian free group.  One of the principal results of this paper is the following:

\begin{thm}\label{t:dichotomy}
Let $\psi\in\Mod(\Sigma,p)$ be a mapping class which acts trivially on the integral homology of $\Sigma$.  Then at least one of the following two occurs:
\begin{enumerate}
\item
There is a finite abelian cover $\Sigma'$ of $\Sigma$ such that $\psi$ acts with spectral radius strictly greater than $1$ on $H_1(\Sigma',\bR)$.  In particular $\psi$ acts with infinite order on $H_1(\Sigma',\bZ)$.
\item
The suspended $3$--manifold $M_{\psi}$ has a fundamental group with positive homology gradient.  In particular, the fundamental group of $M_{\psi}$ is large.
\end{enumerate}
\end{thm}

We will discuss the precise definition of homology gradient later in section \ref{s:large}.
The second part of the dichotomy is evidence for the following important conjecture in $3$--manifold theory:
\begin{conje}
Let $\gam$ be the fundamental group of a hyperbolic $3$-manifold.  Then $\gam$ is large.
\end{conje}

Theorem \ref{t:dichotomy} provides many examples of fibered hyperbolic $3$--manifolds with large fundamental groups.  Indeed, by a famous result of Thurston we have that the suspension of any pseudo-Anosov homeomorphism admits a hyperbolic structure.  On the other hand, it is a well--known fact that any non--central normal subgroup of the mapping class group contains pseudo-Anosov homeomorphisms in all of its cosets.  Finally, it is possible to show that there are nontrivial mapping classes which act trivially on the homology of each finite abelian cover of the base surface.

As for the fundamental groups of fibered $3$--manifolds, it is known that these groups are always large when the monodromy map has finite order or reducible, the latter being a consequence of the work of E. Hamilton in \cite{H}.

Recall that the kernel of the homology representation of the mapping class group is called the {\bf Torelli group}.  The Torelli group acts on the homology of the universal abelian cover of the surface $\Sigma$, where it acts by $\bZ[H_1(\Sigma,\bZ)]$--module automorphisms.  The kernel of this action is called the {\bf Magnus kernel}.  It is not clear a priori that the Magnus kernel is nontrivial.  It turns out that it is actually infinitely generated by a theorem of Church and Farb in \cite{CF}.

\begin{cor}\label{c:magnus}
Let $\psi$ be a mapping class in the Magnus kernel.  Then $\pi_1(M_{\psi})$ is large.
\end{cor}

In order to prove Theorem \ref{t:dichotomy}, we will need to study actions of mapping classes on various infinite covers.  To state the relevant result properly, we recall the so--called {\bf universal $k$--step nilpotent quotients} of $\pi_1(\Sigma)$.  We write $\gamma_1(\Sigma)=\pi_1(\Sigma)$ and \[\gamma_{i+1}(\Sigma)=[\gamma_i(\Sigma),\pi_1(\Sigma)].\]  The sequence of subgroups $\{\gamma_i(\Sigma)\}$ is called the {\bf lower central series} of $\pi_1(\Sigma)$.  The quotient $N_k=\pi_1(\Sigma)/\gamma_{k+1}(\Sigma)$ is called the universal $k$--step nilpotent quotient.  We have a map \[\Mod(\Sigma,p)\to\Aut(N_k),\] whose kernel we write as $J_k$ and call the $k$--th term of the {\bf Johnson filtration}.  Note that the residual torsion--free nilpotence of $\pi_1(\Sigma)$ implies that \[\bigcap_k J_k=\{1\}.\]  For more background on that Johnson filtration, the reader may wish to consult the work of Johnson in \cite{J}.

Write $\Sigma_k$ for the universal $k$--step nilpotent cover of $\Sigma$.  We consider the group $H_1(\Sigma_k,\bZ)$, which is an infinitely generated abelian group.  This group comes equipped with a natural action of $N_k$ by deck transformations, and thus becomes a finitely generated module over $\bZ[N_k]$.  We write this module $M_k=H_1(\Sigma_k,\bZ[N_k])$.  We have an action of $J_k$ on $M_k$ by $\bZ[N_k]$--module automorphisms.

\begin{thm}\label{t:inf}
Let $\psi\in J_k\setminus J_{k+1}$ for some $k\geq 1$.  Then $\psi$ acts with infinite order on $M_{k}$.
\end{thm}

The proof of Theorem \ref{t:inf} is relatively easy, but it does not give infinite order actions of mapping classes on the homology of intermediate finite covers.  We will see that the actions of mapping classes can be ``pro--infinite" but not infinite at any finite stage.  For instance, conjugation by an element of $N_k$ acts with infinite order on $M_k$, but the corresponding automorphism of the homology of a finite intermediate cover will always have finite order.

We note the contrast between Theorems \ref{t:dichotomy} and \ref{t:inf} and the following theorems, the first due to McMullen in \cite{Mc2} and the second due to the author in \cite{K2}:

\begin{thm}
Suppose that $\psi$ is a pseudo-Anosov homeomorphism of a surface with dilatation $K$.  Then either $K$ is the spectral radius of the action of $\psi$ on a finite cover of $\Sigma$, or there is an $0\leq\al<1$ such that \[\sup_{\Sigma'\to\Sigma}K_H(\Sigma')=\al K,\]
where this supremum is taken over all finite covers of $\Sigma$.
\end{thm}

\begin{thm}
Let $\Sigma$ be a surface of hyperbolic type and $\psi$ an mapping class of $\Sigma$.  Suppose that the supremum of the spectral radii of the action of $\psi$ on $H_1(\Sigma',\bR)$ as $\Sigma'$ ranges over finite covers of $\Sigma$ is $K$.  Then the entropy of the action of $\psi$ on any nilpotent quotient of any finite index subgroup of $\pi_1(\Sigma)$ is $\log K$.
\end{thm}

The first of these results shows that the homology of a finite cover of $\Sigma$ may witness the nontriviality of a pseudo-Anosov homeomorphism but generally will not detect its dilatation.  The second theorem shows that there is little to be gained by replacing ``homology of a finite cover" by ``nilpotent quotient of a finite index subgroup".

\section{Acknowledgements}
The author thanks B. Farb, C. McMullen, S. Papadima, A. Reid and A. Suciu for helpful discussions.

\section{The structure of Torelli automorphisms}
In this section we make some basic observations about mapping classes which induce the identity on $H_1(\Sigma,\bZ)$.  Throughout this section, $N$ denotes a nilpotent group and $Z(N)$ denotes its center.

\begin{lemma}\label{l:trans}
Let $G$ be a residually nilpotent group and $1\neq\psi\in\Aut(G)$ an automorphism inducing the identity on $G^{ab}$.  Then there exists a nilpotent quotient $N$ of $G$ and an element $g\in N$ such that $\psi(g)=g\cdot z$, where $1\neq z\in Z(N)$.
\end{lemma}
\begin{proof}
Let $N$ be a nilpotent quotient of $G$ where $\psi(g)g^{-1}$ is separated from the identity.  Since $\psi$ induces the identity on $G^{ab}$, we have that $g$ and $\psi(g)$ are homologous.  Therefore, $\psi(g)g^{-1}$ is deeper in the lower central series of $N$ than $g$ and $\psi(g)$.  Passing to a further quotient of $N$, we may assume that $\psi(g)g^{-1}$ is central.  Since $\psi$ acts trivially on $G^{ab}$, it induces the identity on the center of $N$.  The claim follows.
\end{proof}
\begin{cor}[Compare with \cite{BL}]
If $G$ is residually torsion--free nilpotent and $1\neq\psi\in\Aut(G)$ is an automorphism inducing the identity on $G^{ab}$, then there exists a torsion--free nilpotent quotient $N$ of $G$ and an element $g\in N$ such that $\psi(g)=g\cdot z$, where $1\neq z\in Z(N)$.  In particular, $\psi$ has infinite order.
\end{cor}

Let $G$ be a residually nilpotent group, and define the universal $k$--step nilpotent quotients $\{N_k\}$ of $G$ from the lower central series $\{\gamma_i(G)\}$ of $G$.  We can filter $\Aut(G)$ by an exhausting sequence of subgroups given by \[J_k=ker\{\Aut(G)\to\Aut(N_k)\}.\]  We see that if $\psi\in J_k\setminus J_{k+1}$ for some $k\geq 1$ then $\psi$ acts trivially on $N_k$ but not on $N_{k+1}$, meaning that $\psi$ satisfies the hypotheses of Lemma \ref{l:trans}.

\begin{lemma}\label{l:shift}
Let $G$ be a free group or surface group, and suppose that for some $k\geq 1$ we have that $\psi\in J_k\setminus J_{k+1}$.  Then for each $i\geq 1$, we have that $\psi$ acts on $\gamma_i(G)/\gamma_{i+k+1}(G)$ with infinite order.
\end{lemma}
\begin{proof}
The lemma follows essentially from a degree shift.  Suppose that $g$ is a generator of $G$ such that $\psi$ acts nontrivially on $G$ and which witnesses the fact that $\psi\in J_k\setminus J_{k+1}$.  By definition, $\psi$ acts nontrivially on $N_{k+1}$ and takes some $g\in N_{k+1}$ to $g\cdot z$, where $z$ is a nontrivial central element of $N_{k+1}$.  Furthermore, if $\psi$ does not fix a particular element $h$ of $N_{k+1}$, it perturbs $h$ by a central element $z_h$.  We can thus compute the action of $\psi$ on $[h,g]$ in various nilpotent quotients of $G$.  We see that $[hz_h,gz]$ is equal differs from $[h,g]$ by a product of commutators which can be computed inductively by the rule \[[h,gz]=[h,g][gh^{-1},z].\]

Consider the $\bZ$--Lie algebra $\mL$ of $G$.  When $G$ is a free group then $\mL$ is just a free $\bZ$--Lie algebra.  In this case, $\mL$ has no relations other than the skew--symmetricity and the Jacobi identity.  When $G$ is a surface group then $\mL$ is a quotient of the free $\bZ$--Lie algebra by the ideal generated by the surface relation.  In this case, the ideal is generated by a single homogeneous degree two element.

One can thus check the following: if $g\in\gamma_i(G)\setminus\gamma_{i+1}(G)$ and $h\in\gamma_1(G)\setminus\gamma_2(G)$, then $[h,g]\in\gamma_{i+1}(G)\setminus\gamma_{i+2}(G)$ unless $i=1$ and the homology classes of $g$ and $h$ coincide.  If $i>1$ and $z\in\gamma_k(G)\setminus\gamma_{k+1}(G)$, then we will have $[gh^{-1},z]\in\gamma_{k+1}(G)\setminus\gamma_{k+2}(G)$.  If $i=1$, we will get $[gh^{-1},z]\in\gamma_{k+1}(G)\setminus\gamma_{k+2}(G)$ provided that the homology classes of $g$ and $h$ are distinct.

The lemma can now be proved by induction.  We choose a standard generating basis for $G$ and we let $g$ and $h$ be two generators.  We have $\psi(g)=gz$ and $\psi(h)=hz_h$, where $z$ and $z_h$ are elements of $\gamma_k(G)\setminus\gamma_{k+1}(G)$.  We now consider $G/\gamma_{k+1}(G)=N_{k+2}$.  Inside of $N_{k+2}$, we have that \[[hz_h,gz]=[hz_h,g][gz_h^{-1}h^{-1},z]=[hz_h,g][gh^{-1},z].\]  The rightmost equation can be written as \[[z_h,hg^{-1}][h,g][gh^{-1},z].\]  This will differ in $N_{k+2}$ from $[h,g]$ provided that $[z_h,hg^{-1}]$ is different from $[gh^{-1},z]^{-1}$ in $N_{k+2}$.  This can be arranged by replacing $h$ by $\psi(h)$, for example.  Indeed, if $z_h\neq 0$ then $\psi^2(h)=hz_h^2$ modulo $\gamma_{k+1}(G)$

In general, we will have $g\in\gamma_i(G)\setminus\gamma_{i+1}(G)$, and the same calculations will hold in $N_{i+k+2}$, thus completing the induction.
\end{proof}

\section{Universal torsion--free nilpotent covers}
In this section we will prove Theorem \ref{t:inf}.  We have that $\pi_1(\Sigma_k)=\gamma_{k+1}(\Sigma)$, and that \[H_1(\Sigma_k,\bZ)=\gamma_{k+1}(\Sigma)/[\gamma_{k+1}(\Sigma),\gamma_{k+1}(\Sigma)].\]  Recall that we also have the commutator bracket \[[\cdot,\cdot]:\gamma_i(\Sigma)\times\gamma_j(\Sigma)\to\gamma_{i+j}(\Sigma).\]  It follows that if $g\in\gamma_{k+1}(\Sigma)\setminus\gamma_{2k+2}(\Sigma)$, $g$ will be nontrivial in the abelianization of $\gamma_{k+1}(\Sigma)$.

\begin{proof}[Proof of Theorem \ref{t:inf}]
Let $\psi\in J_k\setminus J_{k+1}$.  Then for each $i$, $\psi$ acts with infinite order on $\gamma_i(\Sigma)/\gamma_{i+k+1}(\Sigma)$, by Lemma \ref{l:shift}.  We have that there exists a $g\in\gamma_i(\Sigma)$ such that the elements $\psi^n(g)$ are all distinct modulo $\gamma_{i+k+1}(\Sigma)$.  In particular, they will all be distinct modulo $[\gamma_i(\Sigma),\gamma_i(\Sigma)]$, provided that $2i\geq i+k+1$, since $[\gamma_i(\Sigma),\gamma_i(\Sigma)]<\gamma_{2i}(\Sigma)$.  In other words, we need $i\geq k+1$.
\end{proof}

Theorem \ref{t:inf} provides, in some sense, an analogue of the Alexander module for a fibered $3$--manifold.  For more details on the Alexander module and Alexander polynomial, see \cite{Mc1}.  Let $M$ be a fibered $3$--manifold with monodromy $\psi$.  We suppose that the action of $\psi$ on a fiber $F$ has at least some invariant cohomology.  Then, one can take a universal torsion--free $\psi$--invariant abelian cover of $F$, say $F'$, and one can study the action of $\psi$ on the homology of $F'$ with coefficients in the group ring over the invariant cohomology.  This group is called the {\bf Alexander module}, and one can extract the {\bf Alexander polynomial} by taking the characteristic polynomial of the action of $\psi$ on the homology of $F'$.  The Alexander polynomial in this case is the analogue of the characteristic polynomial of a Burau or Gassner representation of braid (see Birman's book \cite{Bir}, also \cite{K1}).

\section{Large groups}\label{s:large}
Recall that a group $\gam$ is large if it admits a finite index subgroup with a surjection to a nonabelian free group.  It is an outstanding question in $3$--manifold topology to determine whether or not each hyperbolic $3$--manifold has a large fundamental group.  A positive answer to this question implies a positive answer to the virtually infinite Betti number conjecture, which asserts that for each hyperbolic $3$--manifold $\bH^3/\gam$ and each $n>0$, there is a finite index subgroup $\gam_n<\gam$ such that $\rk H_1(\gam_n,\bQ)\geq n$.  The virtually infinite Betti number conjecture is not known even for fibered hyperbolic $3$--manifolds.  It is known for manifolds which admit a genus $2$ fiber by the work of J. Masters in \cite{Ma}.

In this section we will use the work of M. Lackenby to show that the action of a pseudo-Anosov on the homology of a finite cover of a surface either acquires a rich variety of eigenvalues very quickly, or the associated hyperbolic $3$--manifold will have a large fundamental group.  The relevant tools are {\bf homology rank gradient} and {\bf property ($\tau$)}.  To define these, we fix a prime $p$ and a sequence of nested finite index normal subgroups $\{\gam_i\}$ of a fixed group $\gam$.  We write $d(\gam_i)$ for the rank of $H_1(\gam_i,\bZ/p\bZ)$.  The homology rank gradient of the sequence $\{\gam_i\}$ is \[\gamma=\inf_i\frac{d(\gam_i)}{[\gam:\gam_i]}.\]

If $G$ is a finitely generated group with a finite generating set $S$ and $\{G_i\}$ is a collection of finite index subgroups, we write $X(G/G_i,S)$ for the coset graph of $G_i$ in $G$.  If $A$ is a set of vertices $V$ in a graph, we write $\partial A$ for the set of edges with exactly one vertex in $A$.  The {\bf Cheeger constant} $h(X)$ of a finite graph $X$ is given by \[h(x)=\min_{A\subset V,\, 0<|A|\leq|V|/2}\bigg\{\frac{|\partial A|}{|A|}\bigg\}.\]  $G$ has property ($\tau$) with respect to $\{G_i\}$ if the infimum of $h(X(G/G_i,S))$ is positive for some initial choice of $S$.  Lubotzky's book \cite{Lu} contains a detailed exposition on property ($\tau$).

Lackenby has proved many different results concerning conditions under which a group is large.  The one we will cite here can be found in \cite{La}:
\begin{thm}\label{t:lack}
Let $G$ be a finitely presented group.  Then the following are equivalent:
\begin{enumerate}
\item
$G$ is large.
\item
There exists a sequence of proper nested finite index normal subgroups \[G>G_1>G_2>\cdots\] and a prime $p$ such that:
\begin{enumerate}
\item
$G_{i+1}$ has index a power of $p$ in $G_i$.
\item
$G$ does not have property ($\tau$) with respect to $\{G_i\}$.
\item
$\{G_i\}$ has linear growth of $\pmod p$ homology.
\end{enumerate}
\end{enumerate}
\end{thm}

We begin by ruling out property ($\tau$) for our purposes.

\begin{lemma}\label{l:tau}
Suppose that for each $i$ we have that $G/G_i$ is abelian and that \[G_{\infty}=\bigcap_i G_i\supset [G,G].\]  Then $G$ does not have property ($\tau$) with respect to $\{G_i\}$.
\end{lemma}
\begin{proof}
Since each $G/G_i$ is abelian for each $i$, it is easy to see that the Cheeger constant will arbitrarily small as $i$ tends to infinity.  Note that in a Cayley graph for $G/G_{\infty}$ we may define the Cheeger constants of finite subgraphs by looking at the ratio of the size of the boundary of a finite set to the size of the set itself.  Since $G/[G,G]$ is amenable, the infimum of these Cheeger constants will be zero.  On the one hand, since the intersection of the subgroups $\{G_i\}$ is $G_{\infty}$, for any finite subset of the vertices in a Cayley graph for $G/G_{\infty}$ we may find an $i$ so that this set of vertices is mapped injectively to the set of vertices for a Cayley graph for $G/G_i$.  On the other hand, the degree of each vertex is non--increasing as we project the Cayley graph of $G/G_{\infty}$ to the Cayley graph of $G/G_i$.  It follows that the number of vertices in the boundary of a finite set of vertices in the Cayley graph of $G/G_{\infty}$ cannot increase under the projection map.
\end{proof}

We can now establish some facts about the homology growth of fibered $3$--manifolds.  The homology growth of fibered $3$--manifold groups is very closely related to the action of mapping classes on the homology of finite covers of a surface.  Recall that if \[\Sigma\to M\to S^1\] is a fibered $3$--manifold with monodromy $\psi$, the homology $H_1(M,\bZ)$ is given by $\bZ\oplus F$, where $F$ is the homology of $\Sigma$ which is co--invariant under the action of $\psi$.

Let $\Sigma'\to\Sigma$ be a finite characteristic cover and let $\psi\in\Mod(\Sigma)$, and suppose that $\psi$ commutes with the deck group $G$ of the cover.  We obtain an induced covering $M'\to M$ as follows: the fundamental group of $M$ is given by a semidirect product \[1\to\pi_1(\Sigma)\to\pi_1(M)\to\bZ\to 1.\]  Since the deck group $G$ commutes with the action of $\psi$, we obtain a homomorphism $\pi_1(M)\to G\times\bZ$, and by composing with the projection onto the first factor, a homomorphism $\pi_1(M)\to G$.

Note that the homology contribution of the base space $S^1$ is in the kernel of this homomorphism.  In particular, we obtain a covering map $M'\to M$ which lifts the identity map $S^1\to S^1$.  Furthermore, $\pi_1(M')$ fits into a short exact sequence \[1\to\pi_1(\Sigma')\to\pi_1(M')\to\bZ\to 1,\] where the conjugation action of $\bZ$ on $\pi_1(\Sigma')$ is the restriction of the action of $\psi$.  It follows that the rank of the homology of $M'$ is given by the rank of homology co--invariants of the action of $\psi$ on $H_1(\Sigma',\bZ)$ (equivalently the rank of the $\psi$--invariants of $H^1(\Sigma',\bZ)$).

We can now establish Corollary \ref{c:magnus} before providing a full proof of Theorem \ref{t:dichotomy} from which the Corollary will follow immediately.  We first need to make the following observation:

\begin{lemma}
Let $\psi\in\Mod(\Sigma)$ be a mapping class contained in the Magnus kernel.  Then $\psi$ acts trivially on the integral homology of each finite abelian cover of $\Sigma$.
\end{lemma}
\begin{proof}
Since $\psi$ is in the Magnus kernel, then as an automorphism of $\pi_1(\Sigma)$, $\psi$ acts trivially on the universal metabelian quotient \[\pi_1(\Sigma)/[[\pi_1(\Sigma),\pi_1(\Sigma)],[\pi_1(\Sigma),\pi_1(\Sigma)]].\]  Suppose that $\Sigma'\to\Sigma$ is a finite abelian cover with deck group $A$.  Then we have a short exact sequence \[1\to H_1(\Sigma',\bZ)\to M\to A\to 1.\]  Since $M$ is evidently metabelian, it is a quotient of \[\pi_1(\Sigma)/[[\pi_1(\Sigma),\pi_1(\Sigma)],[\pi_1(\Sigma),\pi_1(\Sigma)]].\]  It follows that $\psi$ restricts to the identity on $H_1(\Sigma',\bZ)$.
\end{proof}

\begin{proof}[Proof of Corollary \ref{c:magnus}]
We suppose that $\Sigma$ is closed with fundamental group $H$.  The proof in the non--closed case is analogous.  Let $p$ be any prime, $g>1$ the genus of $\Sigma$, and $H_i$ be the kernel of the map \[H\to H_1(\Sigma,\bZ/p^i\bZ).\]  It is clear that $\{H_i\}$ forms a sequence of subgroups whose intersection is $[H,H]$ and that $H_i/H_{i+1}$ is abelian.  Furthermore, the rank of $H_1(\Sigma,\bZ/p^i\bZ)$ is $p^{2gi}$.  By an Euler character argument, we have that the genus of the $i^{th}$ surface $\Sigma_i$ corresponding to $H_i$ is $p^{2gi}(g-1)+1$.  It follows that the rank of $H_i^{ab}$ is $p^{2gi}(2g-2)+2$.

Now let $\psi$ be any mapping class in the Magnus kernel and let $M$ be suspension of $\psi$ with fundamental group $G$.  Since $\psi$ is in the Torelli group, $\psi$ commutes with $\pi_1(\Sigma)/H_i$ for each $i$ so that each $H_i$ gives us a finite cover $M_i$ of $M$ with fundamental group $G_i$.  Since $\psi$ is in the Magnus kernel, it follows that the rank of $G_i^{ab}$ is $p^{2gi}(2g-2)+3$.  Note furthermore that the index $[G:G_i]$ is equal to $[H:H_i]$, namely $p^{2gi}$.  Since $g>1$, the ratio between the rank of the homology of $G_i$ and $[G:G_i]$ is bounded away from zero.  Since $G/G_i$ is abelian for all $i$, Lemma \ref{l:tau} implies that $G$ does not have property ($\tau$) with respect to $\{G_i\}$.  By Theorem \ref{t:lack}, $G$ is large.
\end{proof}

\section{Specializing to finite abelian covers}
In this section we will assume that $\Sigma$ is not closed.  The discussion for closed surfaces carries over with minor changes.  Let $\Sigma'\to\Sigma$ be a finite Galois cover with deck group $\gam$.  One would like to describe the homology $H_1(\Sigma',\bC)$ as a module over $\gam$.  This was carried out by Chevalley and Weil in \cite{CW} (see also the discussion in \cite{K1}).  One obtains that $H_1(\Sigma',\bC)$ consists of $d-1$ copies of the regular representation, together with one extra copy of the trivial representation.  We see that $H_1(\Sigma',\bC)$ splits into a direct sum of modules $\{V_{\chi}\}$, where $\chi$ ranges over the irreducible characters of $\gam$.  When $\gam$ is abelian, $V_{\chi}$ either has dimension $d$ or $d-1$, depending on whether $\chi$ is trivial or not.  We will assume henceforth that $\gam$ is abelian.  We write $\pi_{\chi}$ for the projection \[H_1(\Sigma',\bC)\to V_{\chi}.\] 

When $\psi$ is a mapping class inside the Torelli group, we have that $\psi$ acts on each $V_{\chi}$.  We thus obtain one $(d-1)$--dimensional representation of the Torelli group for each nontrivial character of the deck group.  The fundamental point is the following observation, which will allow us to say that the relevant ``finitary part" of $H_1(\widetilde{\Sigma},\bZ[H_1(\Sigma,\bZ)])$ is almost a free $\bZ[H_1(\Sigma,\bZ)]$--module, where here $\widetilde{\Sigma}$ denotes the universal abelian cover of $\Sigma$.

Let $\{\Sigma_i\}$ be a directed system of finite Galois covers of the surface $\Sigma$.  Let $\{G_i\}$ be the associated system of deck groups.  We let $\whG$ be the completed deck group with respect to this directed system.  By taking the complex homology of each element of $\{\Sigma_i\}$, we get a directed system of $\bC$--vector spaces.  We let $H(\Sigma,\bC)$ be the completed homology of this system of finite covers, which is to say the inverse limit of the complex homologies of the covers.  For more details about the construction of this space, the reader is directed to \cite{K1}.  We call $H(\Sigma,\bC)$ the {\bf pro--homology} of $\Sigma$ and its elements are pro--homology classes.  When it is clear, we will not specify the particular system of covers under consideration in the inverse limit.

\begin{lemma}\label{l:basis}
There exist $d-1$ pro--homology classes \[\{c_1,\ldots,c_{d-1}\}\subset H(\Sigma,\bC)\] such that for each finite abelian deck group $\gam$ and each nontrivial irreducible character $\chi$ of $\gam$, \[\{\pi_{\chi}(c_1),\ldots,\pi_{\chi}(c_{d-1})\}\] generate the vector space $V_{\chi}$ over $\bC$.  Furthermore, $\{c_1,\ldots,c_{d-1}\}$ generate a Torelli group--invariant submodule of $H(\Sigma,\bC)$.
\end{lemma}
\begin{proof}
The completed deck group $\whG$ acts on $H(\Sigma,\bC)$ in the natural way.  Namely, we view $H(\Sigma,\bC)$ as a subspace of the vector space \[\prod_{i} H_1(\Sigma_i,\bC)\] and $\whG$ as a subgroup of the product \[\prod_i G_i,\] where the product of the $\{G_i\}$ acts coordinate--wise on the product of the homology vector spaces of the covers.  It is not difficult to verify that this action is well--defined and compatible with the arrows in the directed system $\{\Sigma_i\}$.

On each cover $\Sigma_i\to\Sigma$, we have an explicit description of the action of the deck group $G_i$ on $H_1(\Sigma_i,\bC)$ by Chevalley--Weil theory.  We have that $G_i$ acts via $d-1$ copies of the regular representation, together with one extra copy of the trivial representation.  The regular representations of $G_i$ are free $\bC[G_i]$--modules.

The naturality of the homology functor implies that if \[\Sigma_j\to\Sigma_i\to\Sigma\] is a tower of covers then the induced map $H_1(\Sigma_j,\bC)\to H_1(\Sigma_i,\bC)$ is a $\bC[G_j]$--module homomorphism.  It follows that $H(\Sigma,\bC)$ can be described as $d-1$ copies of the regular representation of $\whG$, together with one extra copy of the trivial representation of $\whG$.

In particular, there are $d-1$ pro--homology classes which generate the $d-1$ regular representations of $\whG$.  Until now our argument has worked for general finite covers, but is now that we require the assumption that each $G_i$ is abelian.  If $\psi$ is any Torelli automorphism then $\psi$ preserves the isotypic components of each finite $\whG$--representation.  It follows that the generators of the $d-1$ regular $\whG$--representation generate a Torelli invariant $\bC[\whG]$--submodule of $H(\Sigma,\bC)$.
\end{proof}

A few words should be said here concerning $\bC[\whG]$ versus $L^2(\whG)$.  In Lemma \ref{l:basis} and in the sequel, showed and will be exploiting the fact that the complex homology of a finite cover is a free module over the deck group, after picking out a trivial representation of the deck group.  In Lemma \ref{l:basis}, we used a homology formulation, and we equally could have used a cohomology formulation.  With the cohomology formulation, we can construct the analogous pro--cohomology of $\Sigma$.  It is rarely necessary to distinguish between pro--homology and pro--cohomology notationally.  It is defined as a direct limit this time, since the arrows go the other way.  To say that $H(\Sigma,\bQ)$ or $H(\Sigma,\bC)$ is a free module over $\bZ[\whG]$ or $\bC[\whG]$ is slightly inaccurate.  Since cohomology comes with a pairing given by considering the intersection form which is preserved by homeomorphisms, and since we are considering a limit of cohomology vector spaces, we should really consider $H(\Sigma,\bC)$ as a module over $L^2(\whG)$.  In the sequel we will be only interested in finite covers, so we may as well restrict our attention to the dense subspace $\bC[\whG]\subset L^2(\whG)$.

Still it remains to see why we can view $\psi$ as a $\bC[\whG]$--module automorphism of the pro--homology or pro--cohomology of $\Sigma$, and are not forced to regard it as an $L^2(\whG)$--module automorphism.  This results from the fact that to describe the action of $\psi$ on the homology of each finite abelian cover of $\Sigma$, we only need a finite amount of data.  To see this, let $\Sigma'$ be any finite abelian cover, and let $c$ be a closed loop on $\Sigma'$ which is the lift of a closed loop on $\Sigma$.  Suppose that $c$ represents a nontrivial homology class on $\Sigma'$.  We have that $c$ represented either a trivial or nontrivial homology class on $\Sigma$.  In the former case, $c$ lifts to a simple closed curve on the universal cover $\widetilde{\Sigma}$, and in the latter case $c$ does not.  In the former case, the action of $\psi$ on the homology of $\widetilde{\Sigma}$ is described by an automorphism of a finitely generated $\bZ[H_1(\Sigma,\bZ)]$--module.  In the latter case, the homology class $[c]$ of $c$ on $\Sigma'$ can be written over $\bZ$ as a sum $[a]+[b]$ of a pullback of a homology class $[a]$ on the base $\Sigma$ and a homology class $[b]$ which dies under the covering map.  In particular, $[b]$ lifts to a homology class on $\widetilde{\Sigma}$.

Since $H_1(\Sigma)\cong\bZ^d$ is the abelianization of $\pi_1(\Sigma)$, the pro--abelian completion of $\pi_1(\Sigma)$ by $\widehat{\bZ}^d$.  If we restrict ourselves to covers of $p$--power order for some prime $p$, we obtain the usual pro--$p$ completion $\bZ_p^d$.  We will write the abelian deck group completion by $\whG$ and will specify whether or not this is a pro--$p$ completing in context if it is unclear.

Now let $\psi$ be a Torelli mapping class, viewed as an automorphism of $\pi_1(\Sigma)$.  We may record the action of $\psi$ on $\{c_1,\ldots,c_{d-1}\}$ and write the answer in terms of $\{c_1,\ldots,c_{d-1}\}$.  For each $i$, there exist $d-1$ elements $p_1^i,\ldots,p_{d-1}^i$ of the group ring $\bC[\whG]$ such that \[\psi(c_i)=\sum_{j=1}^{d-1}p_j^i\cdot c_j.\]

Note that the action of $H_1(\Sigma,\bZ)$ specializes predictably to each $V_{\chi}$.  It follows that the action of $\psi$ on each $V_{\chi}$ can be understood using essentially one matrix, namely the one describing the action of $\psi$ on the rational pro--homology $H(\Sigma,\bQ)$.  To compute the action of $M_{\psi}$ on $V_{\chi}$, we simply compute the image of $H_1(\Sigma,\bZ)$ under the character $\chi$, and we substitute $\chi(t)$ for the action of $t\in H_1(\Sigma,\bZ)$.  The extension of the action to $\whG$ follows by the density of the image of $H_1(\Sigma,\bZ)$ inside of $\whG$.

The matrix $A_{\psi}$ describing the action of $\psi$ has coefficients in $\bC[\whG]$.  We wish to put $A_{\psi}$ in Jordan normal form, so that we can analyze the spectrum of the specialization of $A_{\psi}$ to various finite covers.  We let $F$ be the result of taking the algebraic closure of the field of fractions of of $\bC[\whG]$.  Over $F$, we may put $A_{\psi}$ into Jordan normal form.  We write $\lambda_1,\ldots,\lambda_{d-1}$ for the eigenvalues of $A_{\psi}$.

\begin{lemma}\label{l:one}
Write $A_{\psi}$ as a matrix over the pro--$p$ completed group ring $\bC[\whG]$.  If any eigenvalue $\lambda\in\{\lambda_1,\ldots,\lambda_{d-1}\}$ is a root of unity then the suspended $3$--manifold $M_{\psi}$ has a large fundamental group.
\end{lemma}
\begin{proof}
The proof proceeds very similarly to that of Corollary \ref{c:magnus}.  By raising $A_{\psi}$ to a power, we may assume that $\lambda=1$.  But then on any abelian $p$--power cover with deck group $G$ and in any irreducible $G$--representation, $\psi$ has a fixed vector.  Since $G$ is abelian, there are exactly $|G|-1$ nontrivial irreducible $G$--representations, so that on a cover of order $p^n$, we have that $\psi$ restricts to the identity on a subspace of dimension at least $p^n$.  If $M_{\psi}$ is the suspension of $\psi$, we have that there is a sequence of covers $\{M_i\}$ of $M$ such that the ratio \[\frac{\rk H_1(M_i,\bZ)}{\deg(M_i\to M)}\] is bounded away from zero.  Since each deck group is abelian, the corresponding sequence of covers does not have property ($\tau$).  By Theorem \ref{t:lack}, we obtain the conclusion of the lemma.
\end{proof}

Each eigenvalue of $A_{\psi}$ specializes to an algebraic integer on each intermediate finite cover, since $\psi$ acts on the integral homology of the cover by an integral matrix.  The following elementary fact about algebraic integers will be implicit in the remaining discussion:
\begin{lemma}
Let $\lambda$ be an algebraic integer which lies on the unit circle, and suppose that all the Galois conjugates of $\lambda$ also lie on the unit circle.  Then $\lambda$ is a root of unity.
\end{lemma}

Suppose that on some finite abelian cover of $\Sigma$, an eigenvalue $\lambda$ of $A_{\psi}$ specializes to an algebraic integer which lies off of the unit circle.  Then on that cover, $\psi$ will act with infinite order on the homology, and in fact with spectral radius strictly greater than one (positive homological entropy).

Let \[P(t)=t^{d-1}+a_{d-2}t^{d-2}+\cdots+a_0\] be the characteristic polynomial of $A_{\psi}$, where $a_i\in\bZ[\whG]$ for each $i$.  A priori, each $a_i$ sits in $\bC[\whG]$, but an examination of the proof of Lemma \ref{l:basis} shows that the action of $\psi$ is actually by integral matrices.  If $\Sigma_i$ is a finite intermediate cover with deck group $G$, we can obtain the characteristic polynomial of the action of $\psi$ by considering the polynomials \[P_{\chi}(t)=t^{d-1}+\chi(a_{d-2})t^{d-2}+\cdots+\chi(a_0)\] for each irreducible character $\chi$ and taking the product of the resulting polynomials.

Suppose that for each $\chi$ we have that $A_{\psi}$ acts on $V_{\chi}$ with finite order but that no power of $A_{\psi}$ acts trivially on the homology of each finite abelian cover.  Then we have that $\lambda_1,\ldots,\lambda_{d-1}$ must specialize to roots of unity for each $\chi$ and that the order of these roots of unity must be unbounded.  Each $\lambda_i$ is a root of a fixed monic polynomial $P$ with coefficients in $\bZ[\whG]$.  If $\chi$ is a finite character, we write $\chi(\lambda_i)$ for the corresponding root of $P_{\chi}$.  If $G$ is a finite abelian deck group of a cover $\Sigma_i\to\Sigma$ then the characteristic polynomial of the action of $\psi$ on the homology of $\Sigma_i$ is given by the product \[\prod_{\chi} (t-\chi(\lambda_1))\cdots (t-\chi(\lambda_{d-1})),\] where $\chi$ ranges over the irreducible characters of $G$.  It follows that as the cover gets larger, the orders of $\{\chi(\lambda_i)\}$ get larger as well.

Let $g\in\whG$.  Since $g$ has infinite order, $\bZ[g,g^{-1}]$ is isomorphic to the ring of Laurent polynomials in one variable.  Recall that if $K$ is a field of formal Laurent series over an algebraically closed field, then the algebraic closure of $K$ is given by the field of {\bf Puiseux series} over the base field.  If $K=F((t))$ for some indeterminate $t$, then the Puiseux series over $F$ are given by adjoining fractional powers of the indeterminate $t$, which is to say \[\overline{K}=\bigcup_{n=1}^{\infty}F((t^{1/n})).\]

The matrix $A_{\psi}$ has finitely many entries and hence involves finitely many elements of $\bZ[\whG]$.  Thus any matrix over $\bZ[\whG]$ can be viewed as a matrix over $\bZ[\gam]$, where $\gam$ is a finitely generated subgroup of $\whG$.  Since $\whG$ is abelian and torsion--free, it is isomorphic to $\bZ^n$ for some $n$.  It follows that $A_{\psi}$ can be viewed as a matrix over the integral Laurent polynomial ring \[\bZ[t_1^{\pm 1},\ldots,t_n^{\pm 1}].\]

When we diagonalize $A_{\psi}$, we may interpret the entries as Puiseux series over the algebraic closure of the rational numbers $\overline{\bQ}\subset\bC$ in $n$ indeterminates.  By assumption, the diagonal entries of $A_{\psi}$ are Puiseux series which specialize to roots of unity for every finite character $\chi$.

Let $\lambda$ be one of the diagonal entries of the upper triangular form of $A_{\psi}$.  A priori, \[\lambda=\sum a_{\al_1,\ldots,\al_n}t_1^{\al_1}\cdots t_n^{\al_n},\] where the sum is taken over $n$--tuples $(\al_1,\ldots,\al_n)$ of rational numbers with bounded denominators and each $a_{\al_1,\ldots,\al_n}$ is an algebraic number.  By assumption we have that for each finite character $\chi$, \[\chi(\lambda)=\sum a_{\al_1,\ldots,\al_n}\chi(t_1)^{\al_1}\cdots \chi(t_n)^{\al_n}\] is a root of unity.

\begin{lemma}\label{l:puis}
In the setup above, $\lambda=t_1^{\al_1}\cdots t_n^{\al_n}$ for some rational exponents $\al_1,\ldots,\al_n$.
\end{lemma}

This lemma follows essentially from the work of M. Laurent in \cite{Lau}.  The argument proceeds as follows: the determinant of $A_{\psi}$ is a function of $\{t_1,\ldots,t_n\}$ and also of a new indeterminate $\lambda$.  The vanishing locus $V$ of the determinant determines the eigenvalues of $A_{\psi}$ at various specializations of $\{t_1,\ldots,t_n\}$.  $V$ naturally sits inside of $(\bC^*)^{n+1}$.  Intersecting $V$ with the torus--cross--$\bC^*$ given by restricting $\{t_1,\ldots,t_n\}$ to unit complex numbers, we see that every point of $V$ has a unit complex number for its $\lambda$--coordinate.  In particular, the torsion points on the torus are dense in $V$, so that Laurent's Theorem says that $V$ is a finite union of translates of tori.  It follows that the characteristic polynomial of $A_{\psi}$ is a product of monomials of the form \[(\lambda^k-\overline{t}),\] where $k$ is an integer and $\overline{t}$ is a product of powers of $\{t_1,\ldots,t_n\}$.  The lemma follows immediately from this description of the characteristic polynomial.

We now give a more self--contained proof of the lemma:

\begin{proof}[Proof of Lemma \ref{l:puis}]
Consider the map \[F:(S^1)^n\times S^1\to\bC\] given by taking the $\det(A_{\psi}-\lambda I)$, where $F$ is defined on the first $n$ coordinates by specializing the indeterminates $t_1,\ldots,t_n$ to be roots of unity, and the $\lambda$ is the $(n+1)^{st}$ coordinate.  The vanishing of $F$ defines a subvariety of the torus $(S^1)^{n+1}$ whose properties will give us the conclusion of the lemma.  We will write $X$ for the variety defined by $F=0$.  We may assume without loss of generality that the polynomial $F$ cutting out $X$ is irreducible in the sense that the characteristic polynomial of $A_{\psi}$ is irreducible over $\bZ[t_1^{\pm 1},\ldots,t_n^{\pm1}]$, and we will write $G$ for its Galois group over the algebraic closure of the fraction field of $\bZ[t_1^{\pm1},\ldots,t_n^{\pm1}]$.  Note that $G$ is a finite subgroup of the symmetric group $S_{d-1}$, since $A_{\psi}$ is a polynomial of degree $d-1$ and hence its characteristic polynomial has degree $d-1$.

We first note some basic properties of $F$.  We denote by $Z$ the group of roots of unity inside of $S^1$.  Note that the assumptions of $A_{\psi}$ show that when $(t_1,\ldots,t_n)\in Z^n$, we have that $\lambda\in Z$.  $X$ is naturally equipped with two maps given by the coordinate projections $\pi_1$ and $\pi_2$.  These send $X$ to $(S^1)^n$ and $S^1$ respectively.  We claim that $\pi_1$ is actually either a covering map or constant.  Thus, $X$ is also a torus with coordinates $s_1,\ldots,s_n$, and the map $\pi_1:X\to (S^1)^n$ is given by $s_i^{a_i}=t_i^{b_i}$ for some integers $a_i$ and $b_i$.  Notice that the claim implies the conclusion of the lemma.

To prove that $\pi_1:X\to (S^1)^n=Y$ is a covering map whenever it is nonconstant, we must show that the branches of $X$ lying over a particular point of $Y$ do not come together.  Since $F$ is irreducible, it will generically have $d-1$ roots, counted with multiplicity.  Note that if $p\in Y$ is a rational point (which is to say one whose coordinates are all roots of unity), we have a natural action of $G$ on $\pi_1^{-1}(p)$.  Note that if the characteristic polynomial of $A_{\psi}$ has no repeated roots, then $G$ will act transitively on those roots.

The elements in $\pi_1^{-1}(p)$ can be thought of as sitting inside of the unit circle, and they form a $G$--invariant subset of the group $Z$.  We claim that this implies that elements of $\pi_1^{-1}(p)$ are uniformly separated from each other in the circle.  Notice that since rational points are dense in $Y$, the claim shows that the branches in $X$ lying over a particular point in $Y$ cannot collapse anywhere in $(S^1)^{n+1}$, which implies that $X$ is a fiber bundle over $Y$.  In particular, $\pi_1:X\to Y$ is a covering map.

For each rational point $p\in Y$, $G$ specializes to a subgroup of the Galois group of some cyclotomic extension of $\bQ$.  The Galois group of a cyclotomic extension is cyclic, given by $b\mapsto(\zeta\mapsto\zeta^b)$.  If $\zeta\in\pi_1^{-1}(p)$ and $g\in G$, we have $g^{d-1}(\zeta)=\zeta$.  It follows that $g(\zeta)$ cannot be arbitrarily close to $\zeta$.  Indeed, let $\theta$ denote the argument difference between $\zeta$ and $g(\zeta)$, and we suppose that $g$ is given by exponentiating to the power $b$.  Unwinding the map $z\mapsto z^{b}$ on $S^1$ to $\bR$, we get that the action of $g$ is just multiplication by $b$.  We lift $\zeta$ to $\bR$, where we again call it $\zeta$.  We have that \[(d-1)\cdot b\cdot\zeta\equiv \zeta\pmod 1.\]  If the argument difference between $\zeta$ and $g(\zeta)$ is $\theta$, then we are saying that this is the argument difference between $\zeta$ and $b^k\cdot \zeta$ for some $k$, modulo $1$.  But if $\theta$ is very small when compared to $b^d$, it will not be possible that $b^{d-2}\cdot (b\cdot\zeta)$ to be congruent to $\zeta$ modulo $1$, a contradiction.
\end{proof}

We can now give a proof of Theorem \ref{t:dichotomy}:
\begin{proof}[Proof of Theorem \ref{t:dichotomy}]
Suppose that $\psi$ does not act with spectral radius greater than $1$ on the real homology of a finite abelian cover.  Since the matrix of the action of $\psi$ is integral, it follows that the eigenvalues of the action of $\psi$ are all algebraic integers.  It follows that all the eigenvalues of $\psi$ on each finite abelian cover are roots of unity, so that it must act with finite order on the homology of each finite abelian cover.  Write the matrix $A_{\psi}$ as above and write it in upper triangular form.  If there is a root of unity appearing in the spectrum of $A_{\psi}$, we may replace $\psi$ by a power and assume that there is an occurrence of $1$ in the spectrum of $A_{\psi}$.  Then by Lemma \ref{l:one}, the fundamental group $M_{\psi}$ of the suspension of $\psi$ is large.

Otherwise, each eigenvalue occurring in the spectrum of $A_{\psi}$ is of the form $\lambda=t_1^{\al_1}\cdots t_n^{\al_n}$ by Lemma \ref{l:puis}.  Passing to a power of $A_{\psi}$, we may assume that $\lambda\in\whG$.  Note that replacing $\psi$ by a conjugate by an element of $\pi_1(\Sigma)$ will not change the homotopy type of $M_{\psi}$.  Let $\lambda=g\in\whG$.  For any finite quotient $A$ of $\whG$, we may find an element of $H_1(\Sigma,\bZ)$ whose image in $A$ coincides with that of $g$.  Thus, we may replace $\psi$ by a conjugate and arrange for the action of $\psi$ to have an eigenvalue equal to $1$ on each $V_{\chi}$, where $\chi$ ranges over the irreducible characters of $A$.  Again proceeding as in the proof of Lemma \ref{l:one}, it follows that the homology rank of certain finite index abelian subgroups of $M_{\psi}$ grows at the same rate as their indices.  By Theorem \ref{t:lack}, $\pi_1(M_{\psi})$ is large.
\end{proof}

\section{Some explicit examples of mapping classes and non--inner free group automorphisms which act with infinite order on the homology of a finite cover}
Let $F_n$ denote the free group on $n\geq 3$ generators.  In this section, we exhibit infinite order elements of $\Out(F_n)$ which act with finite order on the homology of every finite index subgroup of $F_n$.  Whenever $n$ is at least $3$, there are certain automorphisms of $F_n$ which are called ``partial conjugations".  An example of such an automorphism of $F_3=\langle a,b,c\rangle$ fixes $a$, sends $ b$ to its conjugate by $a$ and fixes $c$.  Evidently this is a non--inner automorphism of $F_3$ and its image in $\Out(F_3)$ has infinite order.

In general, we call an element $\psi$ of $\Aut(F_n)$ a {\bf partial conjugation} if there is a free generating set $S$ for $F_n$ such that for each $s\in S$, we have $\psi(s)$ and $s$ are conjugate to each other but $\psi$ is not given by conjugation by any element of $F_n$.

\begin{prop}
Let $n>2$ and let $\{x_1,\ldots,x_n\}$ be a free generating set for $F_n$.  Let $\psi\in\Aut(F_n)$ be a partial conjugacy of the form \[(x_1,\ldots,x_n)\mapsto (x_1,x_2^{x_1},x_3,\ldots,x_n).\]  Then there is a finite index subgroup $H$ of $F_n$ such that $\psi$ restricts to an automorphism of $H$, $F_n/H$ is abelian, and the action of $\psi$ on $H^{ab}$ has infinite order.
\end{prop}
\begin{proof}
Consider the map $F_n\to(\bZ/m\bZ)^2$, where $m\geq 3$, given by $x_2\mapsto (1,0)$ and $x_3\mapsto (0,1)$.  Write $H$ for the kernel of this map.  Since $F_n/H$ is abelian, $\psi$ restricts to an automorphism of $H$.  Consider the homology class of the element $[x_2,x_3]\in H$.  Under the action of $\psi$, we have \[[x_2,x_3]\mapsto [x_2^{x_1},x_3].\]  We still have that $x_1$ and all of its conjugates under $(\bZ/m\bZ)^2$ are still nontrivial homology classes.  We can compute the homology class of $\psi([x_2,x_3])$ easily: it is \[-[x_1]+x_2\cdot [x_1]-(x_3x_2)\cdot [x_1]+(x_2^{-1}x_3x_2)\cdot [x_1]+[[x_2,x_3]].\]  Since $F_n/H$ is abelian, this simplifies to \[-[x_1]+x_2\cdot[x_1]-(x_3x_2)\cdot [x_1]+x_3\cdot [x_1]+[[x_2,x_3]].\]  Note that $\psi$ acts trivially on $[x_1]$ and on the deck group.  It follows that if the homology class \[-[x_1]+x_2\cdot[x_1]-(x_3x_2)\cdot [x_1]+x_3\cdot [x_1]\] is nontrivial then $\psi$ has infinite order.  Arranging $m$ sufficiently large guarantees this.
\end{proof}
Note that if $\psi$ is a partial conjugacy of $F_n$ then some power of $\psi$ must act unipotently on the homology of every finite index subgroup of $F_n$.  This follows from the fact that the characteristic polynomial of the action of $\psi$ on the homology of any finite index subgroup of $F_n$ is integral and the fact that the spectral radius of the action gives a lower bound for the entropy of $\psi$ as an automorphism of $F_n$, which is $1$.

In surface mapping class groups, the analogues of partial conjugacies are point--pushing maps.  We will consider point pushing maps in punctured mapping class groups.  Let $\Sigma$ be a punctured surface and let $\gamma$ be an essential, simple closed curve.  Let $T$ be the Dehn twist about $\gamma$, and let $S$ be a twist about a parallel copy $\gamma'$ of $\gamma$ which passes on the other side of the puncture.  Thus, $\Sigma\setminus \gamma\cup\gamma'$ has a component which is a punctured annulus.  The point pushing map of the puncture about $\gamma$ is given by $TS^{-1}$.

\begin{prop}
Let $\gamma$ be an essential, simple closed curve on a surface $\Sigma$, and let $P$ be the point pushing map about $\gamma$.  Then there is a finite cover $\Sigma'\to\Sigma$ to which $P$ lifts and acts with infinite order on $H_1(\Sigma',\bZ)$.
\end{prop}
\begin{proof}
First, we may assume that $\gamma$ is nonseparating.  Indeed, if $\gamma$ is separating then there is a finite cyclic cover of $\Sigma$ where each lift of $\gamma$ is nonseparating (see \cite{K1} for a complete discussion).  The point--pushing map is given by a simultaneous twist about $\gamma$ and a parallel copy $\gamma'$ of $\gamma$ which passes on the other side of the puncture.  Under a cover, the point--pushing map lifts to a simultaneous point--pushing map about the lifts of $\gamma$.

The cyclic cover in the previous paragraph can be taken to be unbranched over the puncture on $\Sigma$.  We may assume that $\Sigma$ has at least two punctures by passing to a cover which is unbranched over the original puncture.  We may therefore restrict our attention to the following situation: we are given two nonseparating curves $\gamma$ and $\gamma'$ which are parallel other than being separated by a puncture $p_1$, and a second puncture $p_2$.  Let $x_1$ and $x_2$ be the homology classes of small loops about $x_1$ and $x_2$.  We obtain a further cover of $\Sigma$ by sending $x_1$ and $x_2$ to $1\in\bZ/2\bZ$ and killing all other homology classes.

This cover which branches over $p_1$ and $p_2$ can be obtained by cutting two copies of $\Sigma$ open along a simple path $c$ connecting $p_1$ and $p_2$ and gluing the two copies together.  By choosing $c$ carefully, we may assume that $c$ intersects only one of $\gamma$ and $\gamma'$, say $\gamma$.  The total preimage of $\gamma$ is a connected curve which enters both copies of $\Sigma$, and the total lift of $\gamma'$ is a union of two simple closed curves, each of which remains in one copy of $\Sigma$.  The square of the point pushing map lifts to a twist about the total preimage of $\gamma$ and a square of a simultaneous twist about the two lifts of $\gamma'$.

It is clear now that there are simple, nonseparating, closed curves on this new cover which intersect $\gamma$ with nontrivial algebraic intersection, but avoid at least one of the lifts of $\gamma'$.  It is easy to see now that the square of the point--pushing map about $\gamma$ lifts to this new cover and acts with infinite order on the homology.
\end{proof}

\end{document}